\documentclass[11pt,a4paper]{article}
\usepackage[OT1]{fontenc}
\usepackage[english]{babel}
\usepackage{amsmath}
\usepackage{amsfonts}
\usepackage{amsthm}

\usepackage{sectsty}
\sectionfont{\small}

\usepackage[bookmarks,pdfpagelabels,linkbordercolor={0.1 0.8 0.4}]{hyperref}

\def\RR{\mathbb{R}}
\def\CC{\mathbb{C}}
\def\ZZ{\mathbb{Z}}

\def\h{ {\cal H} }

\def\b{ {\cal B} }

\def\s{ {\cal S} }

\def\noi{\noindent}

\def\Re{ \mathrm{Re} }
\def\Im{ \mathrm{Im} }

\def\g1{ \mathfrak{g}_1  }

\def\om{\Omega}

\newcommand{\PI}[2]{\left\langle #1 , #2 \right\rangle}

\newtheorem{teo}{Theorem}[section]

\newtheorem{lem}[teo]{Lemma}
\newtheorem{coro}[teo]{Corollary}

\theoremstyle{definition}
\newtheorem{rem}[teo]{Remark}

\begin{document}

%\title{A spectral characterization of normal operator logarithms}
\title{On normal operator logarithms}
\date{}
\author{Eduardo Chiumiento \footnote{Partially supported by Instituto Argentino de Matem\'atica `Alberto P. Calder\'on' and CONICET}}

\maketitle

\abstract{Let $X,Y$ be normal bounded operators on a Hilbert space such that $e^X=e^Y$.  If the spectra of $X$ and $Y$ are contained
in the  strip $\s$ of the complex plane defined by $|\Im(z)|\leq \pi$, we show that $|X|=|Y|$. If $Y$ is only assumed to be bounded, then $|X|Y=Y|X|$. We give a formula for $X-Y$ in terms of spectral projections of $X$ and $Y$ provided that $X,Y$ are normal and $e^X=e^Y$. 
If $X$ is an unbounded self-adjoint operator, which does not have  $(2k+1) \pi$, $k \in \ZZ$,  as eigenvalues, and $Y$ is normal with spectrum in $\s$   satisfying $e^{iX}=e^Y$, then $Y \in \{ \, e^{iX} \, \}''$. We  give alternative proofs and generalizations of results on normal operator exponentials proved by Ch. Schmoeger.

\medskip

{\footnotesize
\noi \textit{AMS classification:  47B15;   47A60.}

\medskip

\noi \textit{Keywords:} Exponential map, normal operator, spectral theorem.}

%\footnote{}

%\setcounter{section}{1} 

\section{Introduction} 

Solutions to the  equation $e^X=e^Y$ were studied by E. Hille \cite{Hille} in the general setting of unital Banach algebras. 
Under the assumption that the spectrum $\sigma(X)$ of $X$ is incongruent (mod $2\pi i$), which means that  $\sigma(X)\cap \sigma(X + 2k\pi i)=\emptyset$ for all $k=\pm 1, \pm 2 , \ldots$, he proved 
that $XY=YX$ and  there exist idempotents $E_1 , E_2 , \ldots , E_n$ commuting with $X$ and $Y$ such that 
\[   X-Y = 2\pi i \sum_{j=1}^n k_j E_j , \, \, \, \, \, \, \,  \, \, \, \, \, \,  \sum_{j=1}^n E_j = I, \, \, \,  \, \, \,  \, \, \, \, \, \, \,  E_i E_j =\delta_{ij},  \]  
where  $k_1, k_2 , \ldots , k_n$ are different integers. If the hypothesis on the spectrum is removed,  it is possible to find non commuting logarithms (see e.g. \cite{Hille, schmoeger laa}). In the setting of Hilbert spaces, when $X$ is a normal operator, the above assumption on the spectrum can be weakened. In fact, Ch. Schmoeger \cite{schmoeger proc}  proved that $X$ belongs to the double commutant of $Y$ provided that $E_X(\sigma(X)\cap \sigma(X + 2k\pi i))=0$, $k=1,2, \ldots$, where $E_X$ is the spectral measure of $X$. We also refer to \cite{paliogiannis} for a generalization of this  result by F. C. Paliogiannis.

In this paper, we study the operator equation $e^X=e^Y$ in the setting of Hilbert spaces under the assumption that the spectra of $X$ and $Y$ belong to a 
non-injective domain of the complex exponential map. Our results include the relation between the modulus of $X$ and $Y$ (Theorem \ref{mod}), a formula for the difference of two normal logarithms in terms of their spectral projections (Theorem \ref{dif}) and  commutation relations when $X$ is a skew-adjoint unbounded operator (Theorem \ref{unbounded}).  The proofs of these results are elementary.  In fact, they  rely on the spectral theorem for normal operators. This approach allows us to give a generalization (Corollary \ref{first thm}) and an alternative proof (Corollary \ref{square}) of two results by Ch. Schmoeger (see \cite{schmoeger laa}).

\section{Notation and preliminaries}

Let $(\h, \PI{\,\cdot\,}{\,\cdot}\,)$   be a complex Hilbert space and $\b(\h)$ be the algebra of bounded operators on $\h$. The spectrum of an operator $X$ is denoted  by $\sigma(X)$, and the set of eigenvalues of $X$ is denoted by $\sigma_p(X)$. The real part of  $X \in \b(\h)$ is  $\Re(X)=\frac{1}{2}(X+X^*)$ and its imaginary part is $\Im(X)=\frac{1}{2}(X-X^*)$. 

If $X$ is a bounded or unbounded normal  operator on $\h$,  we denote by $E_X$  the spectral measure of $X$.  Recall that $E_X$ is defined  on the Borel subsets of $\sigma(X)$, but we may think that $E_X$ is defined on all the 
Borel subsets of $\CC$. Indeed, we can set $E_X(\Omega)=E_X(\Omega \cap \sigma(X))$ for every  Borel set $\Omega \subseteq \CC$. 
Our first lemma is a generalized version of  \cite[Ch. XII Ex. 25]{rudin}, where the normal operator can now be unbounded.

\begin{lem}\label{relation esp}
Let $X$ be a (possibly unbounded) normal operator on $\h$ and $f$ a bounded Borel function on $\sigma(X)$.  Then
\[  
E_{f(X)}(\Omega)=E_X (f^{-1}(\Omega)), \]
for every Borel set $\Omega\subseteq \CC$.
\end{lem} 
\begin{proof}
We define a spectral measure by $E'(\Omega)=E_X(f^{-1}(\Omega))$, where $\Omega$ is any Borel subset of $\CC$.  We are going to show that $E'=E_{f(X)}$. 
Since $f$ is bounded, it follows that $f(X) \in \b(\h)$. Moreover, the operator $f(X)$ is given by
\[  \PI{f(X)\xi}{\eta}=\int_{\CC} f (z)\,dE_{X \, \xi , \eta}(z)\,,  \]
where $\xi, \eta \in \h$  and $E_{X \, \xi , \eta}$ is the complex measure defined by $E_{X \, \xi , \eta}(\Omega)=\PI{E_X(\Omega)\xi}{\eta}$ (see \cite[Theorem 12.21]{rudin}). By the change of measure principle (\cite[Theorem 13.28]{rudin}), we have
\[  \int_{\CC} z \, E'_{\xi , \eta}(z) = \int_{\CC} f(z)\, E_{X \, \xi , \eta} (z).   \]
Therefore $E'$ satisfies the equation 
$  \int_{\CC} z \, E'_{\xi , \eta}(z)= \PI{f(X)\xi}{\eta}, $
%By the uniqueness of the spectral measure of $f(X)$, we find that $E'=E_{f(X)}$.
which uniquely determines the spectral measure of $f(X)$ (see \cite[Theorem 12.23]{rudin}). Hence $E'=E_{f(X)}$.
\end{proof}

The following lemma was first proved in \cite[Corollary 2]{schmoeger laa}. 
See also \cite[Corollary 3]{paliogiannis} for another proof. We give below a  proof for the sake of completeness, which does not depend on
further results of these articles.

\begin{lem}\label{real part}
Let $X$ and $Y$ be normal operators in $\b(\h)$. If $e^X=e^{Y}$, then $\Re(X)=\Re(Y)$.
\end{lem}
\begin{proof}
The following 
computation was done in \cite{schmoeger laa}:
\[   e^{X+X^*}=e^Xe^{X^*}=e^X(e^{X})^*=e^Y(e^Y)^*=e^{Y}e^{Y^*}=e^{Y+Y^*},  \]
where the first and last equalities hold because $X$ and $Y$ are normal.  
Now we may finish the proof  in a different fashion: note that the exponential map, restricted to real axis, has an inverse $\log: \RR_+ \to \RR$. Since $\sigma(X+X^*)\subseteq \RR$ and $\sigma(e^{X+X^*})\subseteq \RR_+$, we can use the continuous functional calculus to get  $X+X^*=\log(e^{X+X^*})=\log(e^{Y+Y^*})=Y+Y^*.$ 
\end{proof}

\noi Throughout this paper, we  use the following notation for subsets of the complex plane:
\begin{itemize}    
\item $\Omega_1 + i \Omega_2 =\{ \, x+ i y  \, : \, x \in \Omega_1 , \, y \in \Omega_2 \,  \}$, where $\Omega_i$, $i=1,2$, are subsets of $\RR$.
\item For short, we write $\RR + i a$  for the set $\RR + i \{ \, a \, \}$.
\item  We write $\s$ for the complex strip 
$\{ \, z \in \CC \,  : \, -\pi\leq \Im(z) \leq \pi \, \},$
and   $\s^{\circ}$ for  the interior of $\s$. 
\end{itemize}

%$\Omega_1 + i \Omega_2 =\{ \, x+ i y  \, : \, x \in \Omega_1 , \, y \in \Omega_2 \,  \}$, where $\Omega_i$, $i=1,2$, are subsets of $\RR$. For short, we write $\RR + i a$  for the set $\RR + i \{ \, a \, \}$.  We will denote by $\s$ the  complex strip 
%$$\s=\{ \, z \in \CC \,  : \, -\pi\leq \Im(z) \leq \pi \, \},$$
%and $\s^{\circ}$ will be the interior of $\s$. 

\begin{lem}\label{normal log}
Let $X, Y$ be normal operators in $\b(\h)$ such that $\sigma(X)\subseteq \s$ and $\sigma(Y)\subseteq \s$. Then $e^X=e^Y$ if and only if the  following  conditions hold:
\begin{enumerate}
\item[i)] $E_X(\Omega)=E_Y(\Omega)$ for all Borel  subsets $\Omega$ of $\s^{\circ}$.
\item[ii)] $\Re(X)=\Re(Y)$.
\end{enumerate}
\end{lem}
\begin{proof}
Suppose that $e^{X}=e^{Y}$. Let $\Omega$ be a Borel measurable subset of $\s^{\circ}$.
By the spectral mapping theorem, 
$$ \sigma(e^X)=\{ \, e^{\lambda} \, : \,\lambda \in \sigma(X) \,  \}=\{ \, e^{\mu} \, : \,\mu \in \sigma(Y) \,  \}=\sigma(e^{Y}).$$
It is well-known that the restriction of the complex exponential map $\exp|_{\s^{\circ}}$ is bijective. Therefore we have
$\sigma(X)\cap \Omega=\sigma(Y) \cap \Omega$, and by Lemma \ref{relation esp}, 
\begin{align*}   
E_X(\Omega) & =E_X(\Omega \cap \sigma(X))= E_X(\,\exp^{-1}(\exp(\Omega \cap \sigma(X)))\,)\\
& = E_{e^X}(\exp(\Omega \cap \sigma(X))) =E_{e^Y}(\exp(\Omega \cap \sigma(Y)))=E_{Y}(\Omega),
\end{align*}
which proves $i)$.  On the other hand,  $ii)$ is proved in Lemma \ref{real part}.

\smallskip 

To prove the converse assertion,  we first note  that
\begin{align} 
E_X(\RR - i \pi) + E_X(\RR + i \pi) & =I-E_X(\s^{\circ})=I-E_Y(\s^{\circ})  \nonumber \\
& =E_Y(\RR - i \pi) + E_Y(\RR + i \pi) \nonumber,   
\end{align}
since  $\sigma(X)\subseteq \s$, $\sigma(Y)\subseteq \s$ and $E_X(\s^{\circ})=E_Y(\s^{\circ})$. Due to the fact that $E_X$ and $E_Y$ coincide on Borel subsets of $\s^{\circ}$, we find that
\[ \int_{\s^{\circ}} e^z \, dE_X (z)=\int_{\s^{\circ}} e^z \, dE_Y (z). \]
Hence we get
\begin{align*}
e^{X}& = \int_{\s} e^z \, dE_X(z) = - \int_{\RR + i \pi} e^{\Re(z)} \, dE_X(z) -   \int_{\RR - i \pi} e^{\Re(z)} \, dE_X(z) + \int_{\s^{\circ}}e^z  \,dE_X(z) \\
& = - e^{\Re(X)}(\,E_X(\RR + i \pi) + E_X(\RR - i \pi)\,) + \int_{\s^{\circ}}e^z \, dE_X(z) \\
& = - e^{\Re(Y)}(\,E_Y(\RR + i \pi) + E_Y(\RR - i \pi)\,) + \int_{\s^{\circ}}e^z \, dE_Y(z)= e^{Y}. \qedhere
\end{align*} 
\end{proof}

\medskip

\begin{rem}\label{line}
We have shown that
$E_X(\RR - i \pi) + E_X(\RR + i \pi)=E_Y(\RR - i \pi) + E_Y(\RR + i \pi)$, 
whenever $X,Y$ are normal bounded operators such that $\sigma(X)\subseteq \s$, $\sigma(Y)\subseteq \s$  and
$e^X=e^Y$.
\end{rem}

\begin{teo}\text{(S. Kurepa \cite{kurepa})}\label{kurepa thm}
Let $X \in \b(\h)$ such that $e^X=N$ is a normal operator. Then 
$$X=N_0 + 2\pi iW,$$ 
where $N_0=\log(N)$ and $\log$ is the 
principal (or any) branch of the logarithm function. The bounded operator $W$ commutes with $N_0$ and there exists a bounded and regular, positive definite self-adjoint operator $Q$ such that $W_0=Q^{-1}WQ$ is a self-adjoint operator the spectrum of which belongs to the set of all integers. 
\end{teo}
%\begin{rem}\label{line}
%In the above proof, we have proved the following fact:   
%$$E_X(\RR - i \pi) + E_X(\RR + i \pi)=E_Y(\RR - i \pi) + E_Y(\RR + i \pi),$$
%whenever $X,Y$ are normal bounded operators such that $\sigma(X)\subseteq \s$, $\sigma(Y)\subseteq \s$  and
%$e^X=e^Y$.
%\end{rem}

\section{Modulus and square of logarithms}

Now we show the relation between the modulus of two normal logarithms with spectra contained in $\s$.

\begin{teo}\label{mod}
Let $X$ be a normal operator in $\b(\h)$. Assume that $\sigma(X) \subseteq \s$ and $e^X=e^Y$. 
\begin{enumerate}
\item[i)] If $Y$ is normal in $\b(\h)$ and $\sigma(Y)\subseteq \s$, then $|X|=|Y|$.
\item[ii)] If $Y \in \b(\h)$, then $|X|Y=Y|X|$. 
\end{enumerate}
\end{teo}
\begin{proof}
$i)$ We will prove that the spectral measures of $|\Im(X)|$ and $|\Im(Y)|$ coincide. Let us set $A=\Im(X)$ and $B=\Im(Y)$. 
Given $\om \subseteq [0,\pi)$, put $\om'=\{ \, x \in \RR \,: \, |x| \in  \om \, \}$. Note that $\RR + i\om' \subseteq \s^{\circ}$. As an application of Lemma \ref{relation esp} and  Lemma \ref{normal log}, we see that
\[   E_{|A|}(\om)=E_A(\om')=  E_X (\RR + i \om')=E_Y (\RR + i \om')=E_B (\om')=E_{|B|}(\om).   \]
By Remark \ref{line}, we have
\begin{align*}
  E_{|A|}(\{ \, \pi \, \} )& =E_A( \{ \, - \pi, \, \pi\, \})=E_X (\RR - i \pi) + E_X(\RR + i \pi )\\
  &=E_Y (\RR - i \pi) + E_Y(\RR + i \pi )=E_{|B|}(\{ \, \pi \, \}). 
\end{align*}
Thus, we have proved  $E_{|A|}=E_{|B|}$, which implies that $|A|=|B|$. On the other hand, by Lemma \ref{normal log}, we know that $\Re(X)=\Re(Y)$. Therefore
$$ |X|^2= \Re(X)^2 + |A|^2 =\Re(Y)^2 + |B|^2 =|Y|^2. $$
Hence $|X|=|Y|$, and the proof is complete.

\smallskip

\noi $ii)$ Since $X$ is a normal operator, $e^X=e^Y$ is also a normal operator. Then by a result by S. Kurepa (see Theorem \ref{kurepa thm}), there exist
operators $N_0$ and $W$ such that $N_0$ is normal, $e^X=e^{N_0}$, $W$ commutes with $N_0$ and $Y=N_0 + 2\pi i W$. In fact, $N_0$ can be defined using the Borel functional calculus by $N_0=\log(e^X)$, where $\log$ is the principal branch of the logarithm. In particular, this implies that $\sigma(N_0)\subseteq \s$. Now we can apply $i)$ to find that $|N_0|=|X|$. Since $N_0W=WN_0$, we have $|N_0|W=W|N_0|$, and this gives $W|X|=|X|W$. Hence $|X|Y=Y|X|$.    
\end{proof}

\medskip

Following  similar  arguments, we can give an alternative proof of a result by Ch. Schmoeger (\cite[Theorem 3]{schmoeger laa}). This result was originally  proved using inner derivations. Note that a minor improvement on the assumption on $\sigma(X)$ over the boundary $\partial \s$ of the strip $\s$ can now be done. Given a set $\om \subseteq \CC$, we denote by $\bar{\om}$ the set $\{ \, x-iy \, : \, x +iy \in \om \, \}$.

\begin{coro}\label{square}
Let $X$ be a normal operator in $\b(\h)$, $\sigma(X)\subseteq \s$, $Y \in \b(\h)$ and $e^X=e^Y$. Suppose that for every Borel subset $\om \subseteq \partial \s \setminus \{ \, -i \pi, \, i \pi\, \}$,  it holds that $E_X(\bar{\om})=0$, whenever $E_X(\om)\neq 0$. Then $X^2Y=YX^2$. 
\end{coro}
\begin{proof}
We will show that   $E_{X^2}(\om_0)$ commutes with $Y$ for every Borel subset $\om_0 \subseteq \sigma(X^2)$. From the equation $e^X=e^Y$, we have  $e^XY=Ye^X$, and thus, 
$E_{e^X}(\om)Y=YE_{e^X}(\om)$ for any Borel set $\om$. Since the set $\om$ is arbitrary, by Lemma \ref{relation esp} we get
\begin{itemize}
\item[1.] $E_{X}(\om')Y=YE_{X}(\om')$ for every subset $\om' \subseteq \s^{\circ}$.
\item[2.] $(E_{X}(\om') + E_{X}(\bar{\om}'))Y=Y(E_{X}(\om') + E_{X}(\bar{\om}'))$, whenever $\om' \subseteq \partial \s$. 
%Here $\bar{\om}'$ denotes the set $\{ \, x - i y \, : \, x + i y \in \om' \,\}$.
\end{itemize}
On the other hand, the image of $\s$ by the  analytic map $f(z)=z^2$ is given by
\[  f(\s)=\{  \, u \pm i 2 t\sqrt{u + t^2}   \,  :   \, u \in [-\pi^2,\infty),  \, u+  t^2  \geq 0 \, \}.  \]
Let us write $f^{-1}(\om_0)=(-\om') \cup \om'$, where $\om'$ is a subset of the half-plane $\Re(z)\geq 0$ and $-\om'$ denotes the set $\{ \, -z \,  : \, z \in \om  \, \}$. 
We need to consider three cases. In the case in which $\om_0 \subseteq f(\s)^{\circ}$, then  $f^{-1}(\om_0)\subseteq \s^{\circ}$. It follows that $E_{X^2}(\om_0)=E_X(-\om') + E_X(\om')$, and by the  item $1.$ above we have $E_{X^2}(\om_0)Y=YE_{X^2}(\om_0)$. In the case where $\om_0 \subseteq \partial f(\s) \setminus \{ \,  \pi^2 \, \}$,   we have either $E_X(\om')=0$ or $E_X(\bar{\om}')=0$  by our assumption on the spectral measure of $X$.  Similarly, it must be either $E(-\om')=0$ or $E_X(-\bar{\om}')=0$. Therefore item 2. above reduces to the desired conclusion, i.e. $E_{X^2}(\om_0)Y=YE_{X^2}(\om_0)$. Finally, if $\om_0=\{ -\pi^2 \}$, then  $E_{X^2}(\om_0)=E_X(\{ \,-i\pi \, \}) + E_X(\{ \,i\pi \, \})$ commutes with $Y$ by item 2., and this concludes the proof. 
\end{proof}

\section{Difference of logarithms}

Let $X, Y$ be normal operators and $k\in \ZZ$. In order to avoid lengthly formulas, let us fix a notation for some  special spectral projections of these operators:
\begin{itemize}
\item $P_{2k+1}=E_X(\,\RR + i(\,(2k-1)\pi,(2k+1)\pi)\,)$;
\item  $Q_{2k+1}=E_Y (\,\RR + i(\,(2k-1)\pi,(2k+1)\pi)\,)$;
\item $E_{2k+1}=E_X(\, \RR + i(2k +1)\pi\,)$;
\item $F_{2k+1}=E_Y(\, \RR + i(2k +1)\pi\,)$.
%\item Projections corresponding to open strips of the form $\RR + i(2(k-1)\pi,2k\pi)$, $k \in \ZZ$, will be denoted by 
%$$P_k=E_X(\RR + i(2(k-1)\pi,2k\pi)), \, \, \, \, \, \, \, \, \, \,  Q_k=E_Y(\RR + i(2(k-1)\pi,2k\pi))$$
%\item Projections corresponding to lines of the form $\RR + i(2k+1)\pi$, $k \in \ZZ$, will be denoted by  
%$$E_k=E_X(\RR + i(2k +1)\pi)),  \, \, \, \, \, \, \, \, \, \, F_k=E_Y(\RR + i(2k +1)\pi)).$$ 
\end{itemize}
As we have  pointed out in the introduction, E. Hille showed that the difference between two logarithms in Banach algebras may be expressed as the sum of multiples of projections   (see \cite[Theorem 4]{Hille}). In order to prove that result,   the spectrum of one of the logarithms must be   incongruent (mod $2\pi i$).      In the case where $X$ and $Y$ are both normal logarithms on a Hilbert space,  the spectral theorem can be used to provide a more general formula.

\begin{teo}\label{dif}
Let $X$ and $Y$ be normal operators in $\b(\h)$ such that $e^X=e^Y$. If $\sigma(X)$ and $\sigma(Y)$  are contained in $\RR + i[ \,(2k_0+1)\pi,(2k_1+1)\pi\,]$ for some $k_0,k_1 \in \ZZ$, then
\[  X-Y = \sum_{k=k_0}^{k_1} 2 k \pi i \, (P_{2k+1} - Q_{2k+1}) + (2k+1)\pi i \,(E_{2k+1} - F_{2k+1}).  \]
\end{teo}
\begin{proof}
We first suppose that $\sigma(X)$  and $\sigma(Y)$ are contained in the strip $\s$. Then we have
%\begin{align*}
$\Im(X)= \Im(X) (  E_X(\s^{\circ}) + E_X(\, \RR + i \pi) + E_X(\, \RR - i \pi) \,)  = \Im (X)P_1 + \pi E_1 - \pi E_{-1}$.
%\end{align*}
Analogously, $\Im(Y)=\Im(Y)Q_1 + \pi F_1 - \pi F_{-1}$. By  Lemma \ref{normal log}, we know that  $\Re(X)=\Re(Y)$ and $E_X(\om)=E_Y(\om)$ for every Borel subset $\om$ of $\s^{\circ}$. It follows that
$$\Im(X)P_1=\int_{\s^{\circ}} \Im(z)\,dE_X(z)=\int_{\s^{\circ}} \Im(z)\,dE_Y(z)=\Im(Y)Q_1, $$
 which implies 
\begin{equation}\label{dif case 1}
 X-Y=  \pi i(E_1 - F_1) -  \pi i (E_{-1} -F_{-1}). 
 \end{equation}
Thus, we have proved the formula in this case. For the general case, without restrictions on spectrum of $X$ and $Y$,  we need to consider the following Borel measurable function
\[    f(t)=\sum_{k=k_0-1}^{k_1} (t - 2k\pi)\, \chi_{\,((2k-1)\pi,(2k+1)\pi]\,}(t), \]
where $\chi_{I}(t)$ is the characteristic function of the interval $I$. Set $A=\Im(X)$ and $B=\Im(Y)$. By Lemma \ref{real part}, $\Re(X)=\Re(Y)$, and since the real and imaginary part of $X$ and $Y$ commute because $X$ and $Y$ are normal, $e^{iA}=e^{X}e^{-\Re(X)}=e^{Y}e^{-\Re(Y)}=e^{iB}$. The function $f$ satisfies $e^{if(t)}=e^{it}$, which implies that  $e^{if(A)}=e^{iA}=e^{iB}=e^{if(B)}$.  Since  $\sigma(f(A))$ and $\sigma(f(B))$ are contained in $[-\pi,\pi]$, we can replace in equation (\ref{dif case 1}) to find that
\begin{align} 
   f(A)-f(B)& =\pi (\,E_{f(A)}(\{ \, \pi \, \}) - E_{f(B)}(\{ \, \pi \, \}\,)   )  \nonumber  \\
    &=  \pi  \sum_{k=k_0-1}^{k_1}  E_A (\{ \, (2k+1)\pi \, \}) -E_B (\{ \, (2k+1)\pi \, \})          \nonumber \\    
     & = \pi \sum_{k=k_0}^{k_1} E_{2k+1} -F_{2k+1} . \label{exp f}
\end{align}
Here we have used Lemma \ref{relation esp} to express $E_{f(A)}, E_A$ and $E_{f(B)}, E_B$ in terms of $E_X$ and $E_Y$ respectively. In particular, note that 
$E_{f(A)}(\{ \, -\pi \, \})=E_{f(B)}(\{ \, -\pi \, \})=0$.  On the other hand, we have
%\begin{align*}
\begin{itemize}
\item[1.] $\displaystyle{f(A) =\sum_{k=k_0-1}^{k_1} (A -2k\pi) \, \chi_{((2k-1)\pi,(2k+1)\pi]}(A)  
= A-  \sum_{k=k_0}^{k_1} 2k\pi (P_{2k+1} + E_{2k+1})}$,
\item[2.] $\displaystyle{f(B)= B-  \sum_{k=k_0}^{k_1} 2k\pi (Q_{2k+1} + F_{2k+1})}$.
\end{itemize}
Therefore 
\begin{align*}
X-Y & =i(A-B) \\
& =  i (f(A)-f(B)) + \sum_{k=k_0}^{k_1} 2k\pi i(P_{2k+1}-Q_{2k+1})  + 2k \pi i(E_{2k+1}- F_{2k+1} ).
\end{align*}  
 Combining this with the expression in (\ref{exp f}), we get the desired formula.
\end{proof}

Below we give a generalization of another result due to Ch. Schmoeger (see \cite[Theorem 5]{schmoeger laa}). The assumptions on the spectrum of $X$ and $Y$ were  more restrictive in \cite{schmoeger laa}: $\|X\|\leq \pi$, $\|Y\|\leq \pi$ and either $-i \pi$ or $i \pi$ does not belong to the point spectrum of one of these operators. However, these hypothesis were necessary to conclude that $X-Y$ is a multiple of a projection; meanwhile $XY=YX$ can be  obtained under more general assumptions (see \cite[Theorem 3]{schmoeger laa}, \cite[Theorem 1.4]{schmoeger proc} and \cite[Theorem 9]{paliogiannis}).

\begin{coro}\label{first thm}
Let $X, Y$ be normal operators in $\b(\h)$. Assume that $\sigma(X) \subseteq \s$, $\sigma(Y)\subseteq \s$ and $e^X=e^Y$. The following assertions hold:
\begin{enumerate}
\item[i)] If $E_1=0$, then $XY=YX$  and $X-Y=-2\pi i \, F_1$.
\item[ii)] If $E_{-1}=0$, then $XY=YX$ and $X-Y=2\pi i \, F_{-1}$.
\item[iii)] If $E_1 = E_{-1}=0$, then $X=Y$.
\end{enumerate}
\end{coro}
\begin{proof}
\noi $i)$ Under these assumptions on the spectra of $X$ and $Y$, we have established that $E_{-1}+ E_1=F_{-1} + F_1$ in Remark \ref{line}. On the other hand, by equation (\ref{dif case 1}) in the proof of Theorem \ref{dif}, we know that
$X-Y=\pi i(E_1 - F_1) -  \pi i (E_{-1} -F_{-1})$.
  Since $E_1=0$, we have $E_{-1}=F_1 + F_{-1}$. It follows that $X=-2 \pi i F_1 + Y$. Hence $X$ and $Y$ commute. 
We can similarly conclude  that  $ii)$ holds true.  To prove $iii)$, note that $E_1=E_{-1}=0$ implies that
$F_1 + F_{-1}=0$, and consequently, $F_1=F_{-1}=0$.  Hence we get $X=Y$.
\end{proof}

\section{Unbounded logarithms}

Let $X$ be a self-adjoint unbounded operator on $\h$. As before, $E_X$ denotes the spectral measure of $X$.   In item $i)$ of our next result, we will give a version of \cite[Theorem 1.4]{schmoeger proc} for unbounded operators (see also \cite[Theorem 9]{paliogiannis}).  To this end, we extend the definition given in \cite{schmoeger proc} for bounded operators: a self-adjoint  unbounded operator $X$ is \textit{generalized $2\pi$-congruence-free} if $$E_X(\sigma(X)\cap \sigma(X+ 2k \pi))=0, \, \, \, \, \, \, \,  k=\pm 1, \pm 2 , \ldots \, .$$
 Given $Y \in \b(\h)$, the commutant of     $Y$ is the set 
\[  \{ \, Y \, \}' = \{  \, Z \in \b(\h) \, : \, ZY=YZ \, \}.   \]
 The double commutant of $Y$ is defined by
\[  \{ \, Y \,\}''= \{ \, W \in \b(\h)  \, : \, WZ=ZW, \text{ for all } Z  \in \{ \, Y \,\}'   \, \}.  \]
If $X$ is a self-adjoint unbounded operator and $Y \in \b(\h)$, recall that $XY=YX$, that is $X$ commutes with $Y$, if $YE_X(\om)=E_X(\om)Y$ for every Borel subset $\om \subseteq \RR$. 
Recall that the exponential $e^{iX}$ of a self-adjoint unbounded operator $X$ is a unitary operator, which can be defined via the Borel functional calculus (see e.g. \cite{rudin}).

\begin{teo}\label{unbounded}
Let   $X$ be a self-adjoint operator on $\h$ and $Y \in \b(\h)$  such that $e^{iX}=e^Y$. 
\begin{enumerate}
\item[i)] If $X$ is generalized $2\pi$-congruence-free,  then $E_X(\om) \in \{ \, Y  \, \}''$ for all Borel subsets $\om$ of $\RR$. In particular, $XY=YX$. 
\item[ii)] If   $\{ \,  (2k+1)\pi \, : k \in \ZZ \, \} \cap \sigma_p(X)$ has at most one element and $Y$ is normal in $\b(\h)$  such that $\sigma(Y) \subseteq \s $, then $XY=YX$.
\item[iii)]  If  $(2k+1) \pi \notin \sigma_p(X)$ for all $k \in \ZZ$ and $Y$ is normal in $\b(\h)$  such that $\sigma(Y) \subseteq \s $, then $Y \in \{ \, e^{iX} \, \}''$. 
\end{enumerate}
\end{teo}
\begin{proof}
$i)$  
Let $Z \in \b(\h)$ such that $ZY=YZ$. It follows that $Ze^{Y}=e^YZ$. Then we have $Ze^{iX}=e^{iX}Z$, and by Lemma \ref{relation esp}, $ZE_{X}(\exp^{-1}(\om))=E_{X}(\exp^{-1}(\om))Z$ for every $\om \subseteq \mathbb{T}$. If $\om'= \exp^{-1}(\om) \cap [-\pi,\pi]$, then
\[  E_{X}(\exp^{-1}(\om))= \sum_{k \in \ZZ} E_X(\om' + 2k  \pi),  \]
where  this series converges  in the strong operator topology. Suppose now that there is some $k \in \ZZ $ such that $E_X(\om' + 2k\pi) \neq 0$. It follows that $\sigma(X) \cap (\om' + 2k\pi) \neq \emptyset$, and $(\om' + 2l\pi) \cap \sigma(X)\subseteq \sigma(X) \cap \sigma(X+2(l-k)\pi)$ for all $l \in \ZZ$. By the assumption on the spectral measure of $X$, $E_X(\om' + 2l  \pi)\leq E_X(\sigma(X) \cap \sigma(X+2(l-k)\pi))=0$ for $l\neq k$. Therefore for each $\om$, the above series reduces to only one spectral projection corresponding to a set of the form $\om' + 2k \pi$. Hence $Z$ commutes with all the spectral projections of $X$.

\smallskip
\noi $ii)$ We need to consider the Borel measurable function $f$ defined in the proof of Theorem \ref{dif}. 
Since $e^{iX}=e^{Y}$, we have that $e^{if(X)}=e^Y$. Recall that $E_X(\{ \, (2k+1)\pi  \, \})\neq0$ if and only if $(2k+1)\pi \in \sigma_p(X)$ (\cite[Theorem 12.19]{rudin}). By the hypothesis on the eigenvalues of $X$,  there is  at most one $n_0 \in \ZZ$ such that $E_X(\{ \, (2n_0+1)\pi  \, \})\neq0$. According to Lemma \ref{relation esp}, we get
\[ E_{f(X)}(\{ \, \pi \, \})= \sum_{k \in \ZZ} E_X(\{ \, (2k+1)\pi \,\} )=E_X(\{ \, (2n_0+1)\pi \,\} ). \]
On the other hand, $E_{f(X)}(\{ \, -\pi  \, \})=0$ for all $k \in \ZZ$ by  definition of the function $f$. According to Corollary \ref{first thm} $ii)$, it follows that $f(X)=Y+2\pi i F_{-1}$.  By Remark \ref{line}, we also know that $E_X(\{ \, (2n_0+1)\pi \,\} )=F_{-1} + F_1$. In order to show that $Y$ commutes with all the spectral projections of $X$, we divide into two cases. If  $\om\subseteq \CC \setminus \{ \, (2k+1)\pi \, : \, k \in \ZZ \, \}$, note that 
$E_X(\om)F_{-1}=0$ because $F_{-1}\leq E_X(\{ \, (2n_0+1)\pi \,\} )$. Hence we get 
$$E_X(\om)Y=E_X(\om)(f(X)-2\pi i F_{-1})=E_X(\om)f(X)=f(X)E_X(\om)=YE_X(\om).$$
If $\om \subseteq \{ \, (2k+1)\pi \, :   \,  k \in \ZZ     \, \}$, we only need to prove that $E_X(\{ \, (2n_0+1)\pi \,\} )$  commutes with $Y$. This follows immediately, because $E_X(\{ \, (2n_0+1)\pi \,\} )$ is the sum of two spectral projections of $Y$. 

\smallskip
\noi $iii)$  As in the proof of $ii)$, we have $e^{if(X)}=e^Y$. Now by the assumption on
 the eigenvalues of $X$, it follows that
\begin{equation}\label{zero sum} 
E_{f(X)}(\{ \, -\pi, \, \pi \, \})= \sum_{k \in \ZZ} E_X(\{ \, (2k+1)\pi \,\} )=0. 
\end{equation}
Applying Corollary \ref{first thm} $iii)$, we get $f(X)=Y$. In particular, $Y$ is a self-adjoint operator such that $\sigma(Y)\subseteq [-\pi,\pi]$.

Let $Z \in \b(\h)$ such that $Ze^{iX}=e^{iX}Z$.  Then we have
$ZE_{e^{iX}}(\om)=E_{e^{iX}}(\om)Z$ for every Borel set $\om \subseteq \mathbb{T}$. We are going to show that $ZE_{Y}(\om')=E_{Y}(\om')Z$ for every $\om'\subseteq [-\pi,\pi]$. We need to consider two cases.  If $\om'\subseteq (-\pi,\pi)$, there exists a unique 
set $\om \subseteq \mathbb{T}\setminus\{ \, -1 \,\}$ such that $\exp^{-1}(\om) \cap [-\pi,\pi]=\om'$. Therefore
\[   E_Y(\om')=E_{f(X)}(\om')=\sum_{k \in \ZZ} E_X(\om' + 2k \pi)=E_X(\exp^{-1}(\om))=E_{e^{iX}}(\om). \]
If $\om'\subseteq \{ \, -\pi, \, \pi \, \}$, by equation (\ref{zero sum}) we find
that $E_Y(\om')=E_{f(X)}(\om')=0$. Hence we obtain that $Z$ commutes with every spectral projection of $Y$. The latter is equivalent to saying that  
$Z$  commute with $Y$, and this concludes the proof.      
\end{proof}

\subsection*{\small{Acknowledgment}}
I would like to thank   Esteban Andruchow  and Gabriel Larotonda for suggesting me to prove  Theorem 
\ref{mod}  $i)$. I am also grateful to them for several helpful conversations.

\medskip

\noi 
{\sc {Departamento de de Matem\'atica, FCE-UNLP,Calles 50 y 115, 
(1900) La Plata, Argentina 
 and 
 Instituto Argentino de Matem\'atica, `Alberto P. Calder\'on', CONICET, Saavedra 15 3er. piso,
(1083) Buenos Aires, Argentina.}     }
                                               
\noi e-mail: {\sf eduardo@mate.unlp.edu.ar}

\end{document}